\documentclass[12pt]{amsart}
\usepackage{amssymb}

\newtheorem{theorem}{Theorem}[section]
\newtheorem{lemma}[theorem]{Lemma}
\newtheorem{proposition}[theorem]{Proposition}
\newtheorem{corollary}[theorem]{Corollary}
\newtheorem{definition}[theorem]{Definition}

\numberwithin{equation}{section}

\begin{document}

\baselineskip=15.5pt

\title[Orthogonal and symplectic parabolic bundles]{Orthogonal and 
symplectic parabolic bundles}

\author[I. Biswas]{Indranil Biswas}

\address{School of Mathematics, Tata Institute of Fundamental
Research, Homi Bhabha Road, Bombay 400005, India}

\email{indranil@math.tifr.res.in}

\author[S. Majumder]{Souradeep Majumder}

\email{souradip@math.tifr.res.in}

\author[M. L. Wong]{Michael Lennox Wong}

\email{wong@math.tifr.res.in}



\date{}

\begin{abstract}
We investigate orthogonal and symplectic bundles with parabolic
structure, over a curve.
\end{abstract}

\maketitle

\section{Introduction}

Let $X$ be an irreducible smooth complex projective curve, and
let $S\, \subset\, X$ be a fixed finite subset. The notion of
parabolic vector bundles on $X$ with $S$ as the parabolic divisor
was introduced by C. S. Seshadri \cite{Se1}. Let $G$ be any
complex reductive group. The generalization of parabolic bundles
to the context of $G$--bundles was done in \cite{BBN}.

Here we take $G$ to be the orthogonal or symplectic group. In these
cases the parabolic bundles can be considered as parabolic vector
bundles with a symmetric or alternating form taking values in a
parabolic line bundle; the form has to be nondegenerate in a
suitable sense.

We define algebraic connection on orthogonal and symplectic
parabolic bundles, and give a criterion for the existence
of such a connection (see Theorem \ref{thm1} and Theorem \ref{thm2}).
This criterion is similar to the one of Weil and Atiyah (\cite{We},
\cite{At}) for the existence of an algebraic connection on a vector 
bundle over $X$.

It turns out that an orthogonal or symplectic parabolic
bundle is semistable (respectively, polystable) if and only if the
underlying parabolic vector bundle is semistable (respectively, 
polystable); see Proposition \ref{prop2}, Proposition \ref{prop3}
and Corollary \ref{cor4}.

We also prove the following theorem (see Theorem \ref{cor3}):

\begin{theorem}
Let $(E_*\, ,\varphi)$ be an orthogonal or symplectic parabolic
bundle. Then $(E_*\, ,\varphi)$ admits an Einstein--Hermitian
connection if and only if it is polystable.
\end{theorem}

\section{Orthogonal and symplectic structure}

\subsection{Parabolic vector bundles}

Let $X$ be an irreducible smooth complex projective curve.
Fix distinct points of $X$
\begin{equation}\label{e1}
S\, :=\, \{x_1\, ,\cdots\, ,x_n\}\, \subset\, X\, .
\end{equation}

Let $E\, \longrightarrow\, X$ be a vector bundle.
A \textit{quasi--parabolic} structure on $E$ over $S$ is a filtration
of subspaces
\begin{equation}\label{e2}
E_{x_i} \,=:\, F_{i,1} \, \supsetneq\, \cdots\, \supsetneq\,
F_{i,j} \, \supsetneq\, \cdots\,
\supsetneq\, F_{i,a_i}\, \supsetneq\, F_{i,a_i+1} \,=\, 0
\end{equation}
over each point of $S$. A \textit{parabolic} structure on $E$ is
a quasi--parabolic structure as above together with real
numbers
\begin{equation}\label{e3}
0\, \leq\, \alpha_{i,1} \, <\, \cdots\, <\, \alpha_{i,j}
\, <\, \cdots\, <\, \alpha_{i,a_i}\, <\, 1
\end{equation}
associated to the quasi--parabolic flags. (See \cite{Se1},
\cite[p. 67]{Se2}, \cite{MY}.)
The numbers in \eqref{e3} are called \textit{parabolic weights}. 

A vector bundle equipped with a quasi--parabolic structure over $S$
and parabolic
weights as above is called a \textit{parabolic vector bundle} with
parabolic structure over $S$. The subset $S$ is called the
\textit{parabolic divisor} for the parabolic vector bundle.

We fix the divisor $S$ once and for all. Henceforth, the
parabolic divisor for all parabolic
vector bundles will be $D$.

For notational convenience, a parabolic vector bundle $(E\, ,
\{F_{i,j}\}\, ,\{\alpha_{i,j}\})$ as above will
also be denoted by $E_*$.

The \textit{parabolic degree} is defined to be
\begin{equation}\label{pd}
\text{par-deg}(E_*)\, :=\, \text{degree}(E)+\sum_{i=1}^n
\sum_{j=1}^{a_i} \alpha_{i,j}\cdot \dim (F_{i,j}/F_{i,j+1})\, .
\end{equation}
The real number 
$\text{par-deg}(E_*)/\text{rank}(E_*)$ is called the
\textit{parabolic slope} of $E_*$, and it is denoted by $\mu_{\rm 
par}(E_*)$.

\subsection{Parabolic dual and parabolic tensor product}

In \cite{MY}, an equivalent definition of parabolic vector bundles
was given. This definition of \cite{MY}
is very useful to work with; we will
recall it now. Take a parabolic vector bundle $(E\, ,
\{F_{i,j}\}\, ,\{\alpha_{i,j}\})$ defined as in \eqref{e2} and
\eqref{e3}. For a point $x_i\, \in\, S$ and $t\, \in\,
[0\, ,1]$, let
$$
E^{i,t}\, \subset\, E
$$
be the coherent subsheaf defined as follows: if $t\, \leq\, 
\alpha_{i,1}$, then
$$
E^{i,t}\, =\, E\, ,
$$
if $t\, >\, \alpha_{i,1}$, then $E^{i,t}$
is defined by the short exact sequence
$$
0\, \longrightarrow\, E^{i,t}\, \longrightarrow\, E
\, \longrightarrow\, E/F_{i,j+1} \, \longrightarrow\, 0\, ,
$$
where $j\, \in\, [1\, ,a_i]$ is the largest number such that 
$\alpha_{i,j}\, < \, t$. Since $F_{i, a_i+1}\,=\, 0$ (see
\eqref{e2}), it follows that $E^{i,t}\, =\, E\bigotimes_{
{\mathcal O}_X}{\mathcal O}_X(-x_i)$ for $t\, >\,
\alpha_{i,a_i}$. For $t\, \in\, [0\, ,1]$, define
$$
E^{(t)}\, =\, \bigcap_{i=1}^n E^{i,t}\, \subset\, E\, .
$$

Now we have a filtration of coherent sheaves
$\{E_t\}_{t\in \mathbb R}$ defined by
\begin{equation}\label{et}
E_t\, :=\, E^{(t-[t])}\otimes {\mathcal O}_X(-[t]S)\, ,
\end{equation}
where $[t]$ is the integral part of $t$ (so $0\, \leq\,
t-[t]\, <\, 1$). Note that
\begin{enumerate}
\item{} the sheaf $E_t$ decreases as $t$ increases,

\item the filtration $\{E_t\}_{t\in \mathbb R}$
is left--continuous, more precisely, there is an $\epsilon\,>\,
0$ such that $E_{t-\epsilon}\,=\, E_t$ for all $t$, and

\item $E_{t+1} \,=\, E_t\bigotimes {\mathcal O}_X(-S)$ for all $t$.
\end{enumerate}
{}From the construction of the filtration $\{E_t\}_{t\in \mathbb R}$
it is evident that the parabolic vector bundle $(E\, ,
\{F_{i,j}\}\, ,\{\alpha_{i,j}\})$ can be recovered from it. Conversely,
given any filtration of coherent sheaves satisfying the above three
conditions, we get a parabolic vector bundle on $X$ with parabolic
parabolic structure over $S$. In \cite{MY}, a parabolic vector
bundle on $X$ with parabolic structure over $S$ is defined to
be a filtration of coherent sheaves satisfying the above three
conditions.

Let
\begin{equation}\label{iota}
\iota\, :\, X\setminus S\, \hookrightarrow\, X
\end{equation}
be the inclusion of the complement. For any coherent sheaf $V$
on $X\setminus S$, the direct image $\iota_*V$ is a quasi--coherent
sheaf on $X$.

Let $\{V_t\}_{t\in \mathbb R}$ and $\{W_t\}_{t\in \mathbb R}$
be the filtrations corresponding to two parabolic
vector bundles $V_*$ and $W_*$ respectively. Consider the 
torsionfree quasi--coherent
sheaf $\iota_*((V_0\bigotimes W_0)\vert_{X\setminus S})$ on $X$,
where $\iota$ is defined in \eqref{iota}. Note that
$V_s\bigotimes W_t$ is a coherent subsheaf of it for
all $s$ and $t$. For any $t\, \in\,
\mathbb R$, let
$$
{\mathcal E}_t\, \subset\, \iota_*((V_0\otimes W_0)\vert_{X
\setminus S})
$$
be the quasi--coherent subsheaf generated by all
$V_\alpha\bigotimes W_{t-\alpha}$, $\alpha\, \in\, \mathbb R$.
It is easy to see that ${\mathcal E}_t$ is a coherent sheaf, and
the collection $\{{\mathcal E}_t\}_{t\in \mathbb R}$ satisfies all
the three conditions needed to define a parabolic vector bundle
on $X$ with parabolic structure over $S$.

The parabolic vector bundle defined by $\{{\mathcal E}_t\}_{t\in \mathbb 
R}$ is denoted by $V_*\bigotimes W_*$, and it is called the
\textit{tensor product} of $V_*$ and $W_*$.

Now consider the torsionfree quasi--coherent
sheaf $\iota_*((V^*_0\bigotimes W_0)\vert_{X\setminus S})$ on $X$.
For any $t\, \in\, \mathbb R$, let
$$
{\mathcal F}_t\, \subset\, \iota_*((V^*_0\otimes W_0)\vert_{X
\setminus S})
$$
be the quasi--coherent subsheaf generated by all
$V^*_\alpha\bigotimes W_{\alpha+t}$, $\alpha\, \in\, \mathbb R$.
This ${\mathcal F}_t$ is a coherent sheaf, and
the collection $\{{\mathcal F}_t\}_{t\in \mathbb R}$ satisfies the
three conditions needed to define a parabolic vector bundle
with parabolic structure over $S$.

The parabolic vector bundle defined by $\{{\mathcal F}_t\}_{t\in
\mathbb R}$ is denoted by ${\mathcal H}om(V_*\, , W_*)$.

Let $L^0_*$ be the trivial line bundle ${\mathcal O}_X$ with trivial 
parabolic structure
(there is no nonzero parabolic weight). Note that the
sheaf for $t\, \in\, \mathbb R$
corresponding to this parabolic line bundle
is ${\mathcal O}_X([-t]D)$, so the
filtration is $\{{\mathcal O}_X([-t]D)\}_{t\in \mathbb R}$.
The \textit{parabolic dual} of $V_*$ is defined to be
$$
V^*_*\, :=\, {\mathcal H}om(V_*\, , L^0_*)\, .
$$
If $\{U_t\}_{t\in \mathbb R}$ is the filtration corresponding to
the parabolic vector bundle $V^*_*$, then
$U_t\, =\, (V_{\epsilon -t-1})^*$, where $\epsilon$ is a
sufficiently small positive real number.

We also note that ${\mathcal H}om(V_*\, , W_*)\,=\, W_*\bigotimes 
V^*_*$. (See \cite{Bis}, \cite{Yo}.)

\subsection{Orthogonal and symplectic parabolic
bundles}\label{se2.3}

Fix a parabolic line bundle $L_*$.

Let $E_*$ be a parabolic vector bundle, and let
\begin{equation}\label{e4}
\varphi\, :\, E_*\otimes E_* \, \longrightarrow\, L_*
\end{equation}
be a homomorphism of parabolic bundles. Tensoring both
sides of the above homomorphism with the parabolic dual $E^*_*$ 
we get a homomorphism
\begin{equation}\label{e5}
\varphi\otimes\text{Id}\, :\, 
E_*\otimes E_* \otimes E^*_*\, \longrightarrow\, L_*\otimes E^*_*\, .
\end{equation}
We note that the sheaf of sections of the vector bundle
underlying $E_* \otimes E^*_*$ is the sheaf of
endomorphisms of $E$ preserving the quasi--parabolic
flags. Sending any locally defined function
$h$ to the locally defined endomorphism of $E$ given by
pointwise multiplication with $h$, the trivial line bundle ${\mathcal 
O}_X$ equipped with the trivial parabolic structure (meaning there
is no nonzero parabolic weight) is realized
as a parabolic subbundle of $E_* \otimes E^*_*$. In fact, 
this line subbundle ${\mathcal O}_X\, \subset\, E_* \otimes E^*_*$ is
a direct summand of the subbundle of $E_* \otimes E^*_*$ defined
by the sheaf of parabolic endomorphisms of trace zero. Let
\begin{equation}\label{e6}
\widetilde{\varphi}\, :\, 
E_* \, \longrightarrow\, L_*\otimes E^*_*
\end{equation}
be the homomorphism defined by the composition
$$
E_*\, =\, E_*\otimes {\mathcal O}_X\, \hookrightarrow\, E_*\otimes
(E_* \otimes E^*_*)\,=\, (E_*\otimes E_*) \otimes E^*_* \, \stackrel{
\varphi\otimes \text{Id}}{\longrightarrow}\, L_*\otimes E^*_*\, .
$$

\begin{definition}\label{def1}
{\rm A} parabolic symplectic bundle {\rm is a pair $(E_*\, ,
\varphi)$ of the above form such that $\varphi$ is
anti--symmetric, and the homomorphism $\widetilde\varphi$
in \eqref{e6} is an isomorphism.}

{\rm A} parabolic orthogonal bundle {\rm is a pair $(E_*\, , 
\varphi)$ of the above form such that $\varphi$ is
symmetric, and the homomorphism $\widetilde\varphi$ is an
isomorphism.}
\end{definition}

\subsection{Equivalence with other definitions in case of rational
parabolic weights}

When all the parabolic weights are rational, principal bundles with
parabolic structure were defined in \cite{BBN}, \cite{BBN1}.
We will show that Definition \ref{def1} coincides with the
definition in \cite{BBN}, \cite{BBN1} when the parabolic weights
are rational.

In this subsection we assume that all the
parabolic weights are rational numbers.

We recall that there is a natural correspondence between parabolic
vector bundles on $X$ and orbifold vector bundles after we fix a 
suitable ramified Galois covering of $X$ depending on the common
denominator of the parabolic weights \cite{Bi2}. This correspondence
takes the dual of a parabolic vector bundle $E_*$ to the usual dual
of the orbifold vector bundle corresponding to $E_*$; it takes the
tensor product of two parabolic vector bundles to the
usual tensor product of the corresponding orbifold vector bundles.

Therefore, if $(E_*\, , \varphi)$ is a parabolic symplectic or
orthogonal vector bundle (see Definition \ref{def1}), then $\varphi$ 
induces a bilinear form $\varphi'$ on the orbifold vector 
bundle $E'$ corresponding to the parabolic vector bundle $E_*$.
This form $\varphi'$ takes values
in the orbifold line bundle $L'$ corresponding to the parabolic 
line bundle $L_*$. We note that $\varphi'$ is nondegenerate because
$\widetilde{\varphi}$ in \eqref{e6} is an isomorphism.
Therefore, $(E'\, , \varphi')$ is an
orbifold symplectic or orthogonal vector bundle depending on
whether $(E_*\, , \varphi)$ is symplectic or orthogonal.
Hence $(E_*\, , \varphi)$ defines a principal parabolic bundle
in the sense of \cite{BBN}, \cite{BBN1} (see \cite[pp. 350--351,
Theorem 4.3]{BBN}, \cite[p. 124, Theorem 1.1]{BBN1}). The converse
also follows similarly.

\subsection{Adjoint bundle}

Let $E_*$ be a parabolic vector bundle over $X$; the vector bundle
underlying this parabolic vector bundle will be denoted by $E$. Let
$U$ be a Zariski open subset of $X$, and let
$$
T\, :\, E\vert_U\, \longrightarrow\, E\vert_U
$$
be an ${\mathcal O}_U$--linear homomorphism. This homomorphism
is called \textit{parabolic} if for every $x_i\, \in\, U\bigcap S$,
$$
T(F_{i,j})\, \subset\, F_{i,j}
$$
for every $j\, \in\, [1\, , a_i]$ (see \eqref{e2}). Let
$$
End^p(E_*)\, \subset\, End(E)\, :=\, E\otimes E^*
$$
be the coherent subsheaf defined by the sheaf of all parabolic
endomorphisms of $E$. As mentioned in Section
\ref{se2.3}, $End^p(E_*)$ is the
vector bundle underlying the parabolic vector bundle $E_*\otimes
E^*_*$.

Now take a homomorphism $\varphi\, :\, E_*\otimes E_*\, 
\longrightarrow\, L_*$ such that $(E_*\, ,\varphi)$ is a symplectic
or orthogonal parabolic bundle. Let $L$ be the line bundle 
underlying the parabolic bundle $L_*$. Since the vector bundle
underlying $E_*\otimes E_*$ contains $E\otimes E$ as a subsheaf, 
the homomorphism $\varphi$ gives a homomorphism
$$
\varphi_0\, :\, E\otimes E \, \longrightarrow\, L\, .
$$
Let
\begin{equation}\label{ad}
\text{ad}(E_*\, ,\varphi)\, \subset\, End^p(E_*)
\end{equation}
be the subbundle generated by the sheaf of homomorphisms
$$
T\, :\, E\, \longrightarrow\, E
$$
lying in $End^p(E_*)$ such that
$$
\varphi_0(T(\alpha)\otimes\beta)+\varphi_0(\alpha\otimes
T(\beta))\,=\, 0
$$
for all locally defined sections $\alpha$ and $\beta$ of $E$.

The vector bundle $\text{ad}(E_*\, ,\varphi)$ defined in \eqref{ad}
will be called the \textit{adjoint} bundle of $(E_*\, ,\varphi)$.

The rank of $\text{ad}(E_*\, ,\varphi)$
is the dimension of the orthogonal or symplectic group
corresponding to $\varphi$. Let
$$
End^0(E)\, \subset\ End(E)
$$
be the subbundle of corank one defined
by the sheaf of endomorphisms of $E$ of trace zero.
Since $\varphi\vert_{X\setminus S}$ is nondegenerate, we have
$$
\text{ad}(E_*\, ,\varphi)\vert_{X\setminus S}\, \subset\,
End^0(E)\vert_{X\setminus S}\, .
$$
Since $X\setminus S$ is Zariski open, this implies that
\begin{equation}\label{in}
\text{ad}(E_*\, ,\varphi)\, \subset\,End^0(E)\, .
\end{equation}

We will give another description of the subbundle $\text{ad}(E_*\, 
,\varphi)$. For any locally defined parabolic homomorphism 
$T\, :\, E_*\, \longrightarrow\, E_*$, we have the dual homomorphism
$$
T^*\, :\, E^*_*\, \longrightarrow\, E^*_*\, .
$$
The section $T$ of $End^p(E_*)$ lies in $\text{ad}(E_*\, ,\varphi)$
if and only if the following diagram is commutative
$$
\begin{matrix}
E_* & \stackrel{\widetilde\varphi}{\longrightarrow} & E^*_*\otimes L_*\\
~\Big\downarrow T &&
~\,~\,~\,~\,~\,~\,~\,~\,\Big\downarrow   T^*\otimes\text{Id}_L\\
E_* & \stackrel{\widetilde\varphi}{\longrightarrow} & E^*_*\otimes L_*
\end{matrix}
$$
where $\widetilde\varphi$ is the isomorphism in \eqref{e6}.

Equip the line bundle ${\mathcal O}_X(S)$ with the trivial parabolic
structure (so there is no nonzero
parabolic weight). Consider the parabolic vector bundle 
$\text{ad}(E_*\, ,\varphi)\otimes{\mathcal O}_X(S)$. Let
\begin{equation}\label{ad2}
\text{ad}^0(E_*\, ,\varphi)\, \subset\, \text{ad}(E_*\, ,\varphi)
\otimes{\mathcal O}_X(S)
\end{equation}
be the coherent subsheaf defined by the sheaf of all locally
defined sections 
$$
T\, :\, E_*\, \longrightarrow\, E_*\otimes {\mathcal O}_X(S)
$$
of $\text{ad}(E_*\, ,\varphi)\otimes{\mathcal O}_X(S)$ such that
$T(F_{i,j})\, \subset\, F_{i,j+1}$ for all $x_i\,\in\, S$ in the domain
of $T$ and all $j\, \in\, [1\, ,a_i]$ (see \eqref{e2}).

Let $tr \, :\, End(E)\otimes End(E)\, \longrightarrow\, {\mathcal O}_X$ 
be
the homomorphism defined by $A\otimes B\, \longmapsto\,
\text{trace}(A\circ B)$. Consider the composition
$$
\text{ad}(E_*\, ,\varphi)\otimes \text{ad}^0(E_*\, ,\varphi)
\,\hookrightarrow\, End(E)\otimes (End(E)\otimes{\mathcal O}_X(S))
\, \stackrel{tr\otimes \text{Id}}{\longrightarrow}\,
{\mathcal O}_X(S)\, .
$$
The image of this composition homomorphism
is the subsheaf ${\mathcal O}_X\,
\subset\, {\mathcal O}_X(S)$, and the above pairing 
$$
\text{ad}(E_*\, ,\varphi)\otimes \text{ad}^0(E_*\, ,\varphi)
\,\longrightarrow\, {\mathcal O}_X
$$
is nondegenerate. Hence we get an isomorphism of vector bundles
\begin{equation}\label{g1}
\text{ad}^0(E_*\, ,\varphi)\, \stackrel{\sim}{\longrightarrow}\,
\text{ad}(E_*\, ,\varphi)^*\, .
\end{equation}

\section{Connection on parabolic orthogonal and symplectic 
bundles}

\subsection{Logarithmic connection}

Let $\Omega_X$ be the canonical line bundle of $X$.
A \textit{logarithmic
connection} on a holomorphic vector bundle $W\,\longrightarrow\,X$ 
singular 
over $S$ is a first order holomorphic differential operator
$$
D\, :\, W\, \longrightarrow\,W\otimes \Omega_X\otimes{\mathcal O}_X(S)
$$
satisfying the Leibniz identity which says that
$D(fs)\,=\, fD(s)+ s\otimes df$, where $f$ is a locally defined
holomorphic function and $s$ is a locally defined holomorphic
section of $W$. See \cite{De} for the details.

Since $S$ is fixed, we will often refrain from referring to $S$.
A logarithmic connection on $X$ will always mean that the singular
locus of the logarithmic connection is contained in $S$.

For notational convenience, the line bundle $\Omega_X\otimes
{\mathcal O}_X(S)$ will be denoted by $\Omega_X(S)$.
Take any point $x_i\, \in\, S$. Using the Poincar\'e adjunction formula,
the fiber $\Omega_X(S)_{x_i}$ is identified with $\mathbb 
C$. We recall that if
$f$ is a function defined on an open neighborhood of $x_i$ such
that $f(x_i)\,=\, 0$, and $df(x_i)\, \not=\, 0$, then the evaluation
of the section $(df)/f$ of $\Omega_X(S)$ at $x_i$ is $1$. For a
logarithmic connection $(V\, ,D)$, consider the composition
\begin{equation}\label{ec}
V\, \stackrel{D}{\longrightarrow}\, V\otimes 
\Omega_X(S)\,\longrightarrow
\, (V\otimes \Omega_X(S))_{x_i}\, =\, V_{x_i}\, ,
\end{equation}
where $V\otimes\Omega_X(S)\,\longrightarrow\, (V\otimes \Omega_X
(S))_{x_i}$ is the restriction map. This composition
is ${\mathcal O}_X$--linear due to the Leibniz identity, hence it
defines an endomorphism of the complex vector space $V_{x_i}$. This 
endomorphism is denoted by
$$
\text{Res}(D,x_i)\, \in\, \text{End}_{\mathbb C}(V_{x_i})\, ,
$$
and it is called the \textit{residue} of $D$ at $x_i$;
see \cite[p. 53]{De}.

Let $E_*\,=\, (E\, ,\{F_{i,j}\}\, ,\{\alpha_{i,j}\})$ be a parabolic 
vector bundle. An \textit{algebraic connection} on
$E_*$ is a logarithmic connection $D$ on $E$ such that
for all $x_i\, \in\, S$, and all $j\, \in\, [1\, ,a_i]$
(see \eqref{e2} for $a_i$), the following two conditions hold:
\begin{equation}\label{cd1}
\text{Res}(D,x_i)(F_{i,j}) \, \subseteq\, F_{i,j}
\end{equation}
(this condition implies
that $\text{Res}(D,x_i)$ induces an endomorphism of the
quotient vector space $F_{i,j}/F_{i,j+1}$), and
\begin{equation}\label{cd2}
\text{Res}(D,x_i)\vert_{F_{i,j}/F_{i,j+1}}\,=\,
\alpha_{i,j}\cdot \text{Id}_{F_{i,j}/F_{i,j+1}}\, ,
\end{equation}
where $\alpha_{i,j}$ is the parabolic weight in \eqref{e3}.

\begin{lemma}\label{lem1}
If $E_*$ admits an algebraic connection, then $\text{\rm 
par-deg}(E_*)\,=\, 0$.
\end{lemma}

\begin{proof}
Let $W$ be a vector bundle over $X$ equipped with a logarithmic
connection $D$ singular over $S$. Then
$$
\text{degree}(V)+\sum_{i=1}^n
\text{trace}(\text{Res}(D,x_i))\,=\, 0
$$
\cite[p. 16, Theorem 3]{Oh}. In view of this, the lemma follows
immediately from \eqref{cd2} and the definition of parabolic
degree given in \eqref{pd}.
\end{proof}

\begin{lemma}\label{lem0}
Let $E_*$ and $F_*$ be parabolic vector bundles equipped
with algebraic connections $D_E$ and $D_F$ respectively. Then
$D_E$ and $D_F$ together induce an algebraic connection on
$E_*\otimes F_*$. Also, $D_E$ induce an algebraic connection
on $E^*_*$.
\end{lemma}

\begin{proof}
Let $E$ and $F$ be the vector bundles underlying $E_*$ and
$F_*$ respectively.
The logarithmic connections $D_E$ on $E$ induces a logarithmic 
connection on the dual vector bundle $E^*$; this logarithmic
connection on $E^*$ will be denoted by $D'_E$.
Let $E^*_0$ be the vector bundle underlying the parabolic vector
bundle $E^*_*$. This $E^*_0$ is a subsheaf of $E^*$.
It is straight--forward to check that
the logarithmic connection $D'_E$ on $E^*$
produces a logarithmic connection on $E^*_0$, and the resulting
logarithmic connection on $E^*_0$ is an algebraic connection on the
parabolic vector bundle $E^*_*$.

The two logarithmic connections $D_E$ and $D_F$ together define 
a logarithmic connection on $E\otimes F$.
The line bundle ${\mathcal O}_X(S)$ has a
natural logarithmic connection given by the de Rham differential
$f\, \longmapsto\, df$. This logarithmic connection on
${\mathcal O}_X(S)$ and the
above logarithmic connection on $E\otimes F$
together define a logarithmic connection on $E\otimes F\otimes
{\mathcal O}_X(S)$.

The vector bundle $(E_*\otimes F_*)_0$ underlying the parabolic
tensor product $E_*\otimes F_*$ is a subsheaf of $E\otimes F\otimes
{\mathcal O}_X(S)$. It is straight--forward to check that the
above logarithmic connection on $E\otimes F\otimes
{\mathcal O}_X(S)$ produces a logarithmic connection on
$(E_*\otimes F_*)_0$, and the resulting logarithmic 
connection on $(E_*\otimes F_*)_0$ is an algebraic connection on the
parabolic vector bundle $E_*\otimes F_*$.
\end{proof}

It is known that a parabolic line bundle $E_*$ on $X$ of degree zero
admits an algebraic connection. In fact $E_*$
has a unique unitary flat connection \cite{Si}, \cite{Po}, \cite{Bi}.
We include a simple proof.

\begin{lemma}\label{lem2}
Let $E_*$ be a parabolic line bundle with $\text{\rm par-deg}
(E_*)\,=\, 0$. Then $E_*$ admits an algebraic connection.
\end{lemma}

\begin{proof}
For any $x_i\, \in\, S$, let
$0\, \leq\, \lambda_i\, < \,1$ be the parabolic weight
of $E_*$ over $x_i$. Let $d$ be the degree of the vector
bundle $E$ underlying $E_*$. So,
\begin{equation}\label{d}
\text{par-deg}(E_*)\,=\, d +\sum_{i=1}^n \lambda_i\,=\, 0\, .
\end{equation}
An algebraic connection on $E_*$ is a logarithmic connection on
$E$ with residue $\lambda_i$ at each point $x_i\,\in\, S$.

Fix a divisor $\Delta_E\, =\, \sum_{j=1}^{m+d}y_j - \sum_{k=1}^{m} z_k$
such that $E\,=\, {\mathcal O}_X(\Delta_E)$. We may, and we will,
assume that $x_i$, $y_j$ and $z_k$ are all distinct points.
As mentioned before, ${\mathcal O}_X(\Delta_E)$ has a
tautological logarithmic connection ${\mathcal D}_0$ defined by the
de Rham differential defined by $f\, \longmapsto\, df$. This
logarithmic connection ${\mathcal D}_0$ is singular over the
points $\{y_j\}_{j=1}^{m+d}$ and $\{z_k\}_{j=1}^{m}$, and its
residue over each $y_j$ is $-1$ and its residue over each $z_k$ is
$1$. To prove that $E_*$ admits an algebraic connection, it suffices to 
produce a section $\omega$ of the line bundle
\begin{equation}\label{lb}
{\mathcal L}\, :=\,
\Omega_X\otimes {\mathcal O}_X(S)\otimes {\mathcal O}_X(\sum_{j=1}^{m+d}
y_j) \otimes {\mathcal O}_X(\sum_{k=1}^{m} z_k)
\end{equation}
such that the residue of $\omega$ over every $x_i\, \in\, S$ is 
$\lambda_i$, over each $y_j$ is $1$ and over each $z_k$ is $-1$.
Indeed, the logarithmic connection ${\mathcal D}_0+\omega$
on $E$, where $\omega$ is a section of the line bundle
${\mathcal L}$ in \eqref{lb}
satisfying the above residue conditions, is an algebraic
connection on $E_*$.

To construct such a section $\omega$,
consider the short exact sequence of coherent sheaves on $X$
$$
0\,\longrightarrow\, \Omega_X\,\longrightarrow\,{\mathcal L}
\,\longrightarrow\, {\mathcal L}\vert_{S+\sum_{j=1}^{m+d}
y_j+\sum_{k=1}^{m} z_k} \,\longrightarrow\, 0\, .
$$
Let
\begin{equation}\label{le}
H^0(X,\, {\mathcal L}) \,\longrightarrow\,
(\bigoplus_{x_i\in S} {\mathcal L}\vert_{x_i})\oplus
(\bigoplus_{j=1}^{m+d} {\mathcal L}\vert_{y_j})\oplus
(\bigoplus_{k=1}^{m} {\mathcal L}\vert_{z_k})
\,\longrightarrow\, H^1(X,\, \Omega_X)\,=\, \mathbb C
\end{equation}
be the corresponding long exact sequence of cohomologies. From
\eqref{le} it follows that ${\mathcal L}$ has a section with
the given residues over $S$, $\{y_j\}_{j=1}^{m+d}$ and
$\{z_k\}_{j=1}^{m}$ as long as the sum of all the residues is zero.
Therefore, from \eqref{d} we conclude that there
is a section $\omega$ such that the residue of $\omega$ over each
$x_i\, \in\, S$ is
$\lambda_i$, over each $y_j$ is $1$ and over each $z_k$ is $-1$.
This completes the proof of the lemma.
\end{proof}

\subsection{Definition of an algebraic connection}\label{sec3.2}

Let $E_*\,=\, (E\, ,\{F_{i,j}\}\, ,\{\alpha_{i,j}\})$ be a parabolic
vector bundle of rank $r$. Let $L_*$ be a parabolic line bundle.
Let
$$
\varphi\, :\, E_*\otimes E_*\, \longrightarrow\, L_*
$$
be an orthogonal or symplectic parabolic structure. We have
$$
\text{par-deg}(L_*\otimes E^*_*)\,=\,
r\cdot \text{par-deg}(L_*)+ \text{par-deg}(E^*_*)\,=\,
r\cdot\text{par-deg}(L_*)- \text{par-deg}(E_*)\, .
$$
Hence from the
isomorphism $\widetilde{\varphi}$ in \eqref{e6} it follows that
\begin{equation}\label{iii}
r\cdot \text{par-deg}(L_*)\,=\, 2\cdot \text{par-deg}(E_*)\, .
\end{equation}
Therefore,
\begin{equation}\label{ii}
\text{par-deg}(L_*)\,=\, 0 \,\Longleftrightarrow\, 
\text{par-deg}(E_*)\,=\, 
0\, .
\end{equation}

In this subsection we assume that $\text{par-deg}(L_*)
\,=\, 0$.

Since $\text{par-deg}(L_*)\,=\, 0$, from Lemma \ref{lem2} we know
that $L_*$ has an algebraic connection. Fix an algebraic connection 
$D_L$ on $L_*$.

As before, let $(E_*\, , \varphi)$ be a symplectic or orthogonal
parabolic bundle. Let $D$ be an algebraic
connection on the parabolic vector bundle $E_*$.
The algebraic connection $D$ on $E_*$ induces an algebraic
connection on $E^*_*$. This induced algebraic connection
on $E^*_*$ and the algebraic connection $D_L$ on $L_*$
together produce an algebraic connection on $E^*_*\otimes L_*$
(see Lemma \ref{lem0}). 

\begin{definition}\label{def2}
{\rm The algebraic connection on $D$ on $E_*$ is said to be}
compatible {\rm with $\varphi$ if the isomorphism
$\widetilde{\varphi}\, :\, E_*\, \longrightarrow\, L_*\otimes E^*_*$
in \eqref{e6} takes the algebraic connection $D$
on $E_*$ to the algebraic connection on $L_*\otimes E^*_*$
constructed above from $D$ and $D_L$.}

{\rm An} algebraic connection {\rm on $(E_*\, , \varphi)$ is an
algebraic connection on $E_*$ compatible with $\varphi$}.
\end{definition}

Let $D$ be an algebraic connection on the parabolic vector
bundle $E_*$. We will describe a
criterion for $D$ to be compatible with $\varphi$. 

Let $D'$ be the algebraic connection on $E_*\otimes E_*$
induced by $D$ (see Lemma \ref{lem0}). The algebraic connection 
$D$ is compatible with $\varphi$ if and only if the homomorphism 
$\varphi$ intertwines $D'$ and $D_L$ (the given algebraic
connection on $L_*$).

The algebraic connections $D$ and $D_L$ together produce an
algebraic connection
on the parabolic tensor product $L_*\otimes E^*_*\otimes E^*_*$.
On the other hand, $\varphi$ defines a section of
$L_*\otimes E^*_*\otimes E^*_*$. The homomorphism $\varphi$
intertwines $D'$ and $D_L$ if and only if the section of
$L_*\otimes E^*_*\otimes E^*_*$ defined by $\varphi$
is flat with respect to the algebraic
connection on $L_*\otimes E^*_*\otimes E^*_*$ constructed
using $D$ and $D_L$.

If $D_1$ and $D_2$ are algebraic connections on $(E_*\, , \varphi)$, 
then
$$
D_1-D_2\,\in\, H^0(X,\, \text{ad}^0(E_*\, ,\varphi)
\otimes\Omega_X)\, ,
$$
where $\text{ad}^0(E_*\, ,\varphi)$ is the vector bundle
in \eqref{ad2}. Conversely, for any algebraic connection $D$ on
$(E_*\, , \varphi)$, and for any
$$
\theta\, \in\, H^0(X,\, {\rm ad}^0(E_*\, ,\varphi)\otimes\Omega_X)\, ,
$$
their sum
$D+\theta$ is also an algebraic connection on $(E_*\, , \varphi)$. 
Therefore, the following holds:

\begin{lemma}\label{lem3}
The space of all algebraic connections on $(E_*\, , \varphi)$ is an
affine space for the vector space $H^0(X,\, {\rm ad}^0(E_*\, 
,\varphi)\otimes\Omega_X)$.
\end{lemma}

\subsection{Generalized logarithmic connection}

Let $U\, \subset\, X$ be a nonempty
Zariski open subset. The intersection
$U\bigcap S$ will be denoted by $S_U$. The canonical line bundle
of $U$ will be denoted by $\Omega_U$. Fix an algebraic function $w$
on $U$. Let $V$ be an algebraic vector bundle on $U$.

A \textit{generalized logarithmic connection on} $V$
\textit{with weight $w$} is an algebraic differential operator
$$
D\, :\, V\, \longrightarrow\, V\otimes \Omega_U(S_U)\, :=\,
V\otimes \Omega_U\otimes{\mathcal O}_U(S_U)
$$
satisfying the identity
\begin{equation}\label{f4}
D(fs)\,=\, fD(s)+ w{\cdot}s\otimes df\, ,
\end{equation}
where $f$ is a locally defined algebraic 
function, and $s$ is a locally defined algebraic section of $V$.

The identity in \eqref{f4} implies that the order of
the differential operator $D$ is at most one. The
order of $D$ is zero, meaning $D$ is ${\mathcal O}_U$--linear,
if and only if $w\, =\, 0$. We also note that $D$ is a logarithmic
connection over $U$ if $w\,=\,1$.

A \textit{generalized logarithmic connection} on $V\,\longrightarrow
\, U$ is a pair $(w\, , D)$, where $w$ is an algebraic function on $U$, 
and $D$ is a generalized logarithmic connection on $V$ with weight $w$.

Given any generalized logarithmic connection $(w\, , D)$ on
$V\, \longrightarrow\, U$,
for any point $x_i\, \in\, S_U$, consider the composition
$$
V\, \stackrel{D}{\longrightarrow}\, V\otimes\Omega_U\otimes
{\mathcal O}_U(S_U)\, =: \,V\otimes 
\Omega_U(S_U)\,\longrightarrow
\,(V\otimes \Omega_U(S_U))_{x_i}\, =\, V_{x_i}
$$
as in \eqref{ec}. It defines an endomorphism
$$
\text{Res}(D,x_i)\, \in\, \text{End}_{\mathbb C}(V_{x_i})\, ;
$$
this endomorphism will be
called the \textit{residue} of $D$ at $x_i$.

Let $E_*\,=\, (E\, ,\{F_{i,j}\}\, ,\{\alpha_{i,j}\})$ be a
parabolic vector bundle. Let $U\, \subset\, X$ be a
nonempty Zariski open subset.
A \textit{generalized algebraic connection} on $E_*\vert_U$
is generalized logarithmic connection $(w\, ,D)$ on $E\vert_U$ 
satisfying the following condition:
for any $x_i\, \in\, U\bigcap S$,
$$
\text{Res}(D,x_i)(F_{i,j}) \, \subseteq\, F_{i,j}~\,~\,
\text{~and~} ~\,~\, \text{Res}(D,x_i)\vert_{F_{i,j}/F_{i,j+1}}\,=\,
w(x_i)\cdot \alpha_{i,j}\cdot \text{Id}_{F_{i,j}/F_{i,j+1}}
$$
for all $j\, \in\, [1\, ,a_i]$; the first condition ensures that
$\text{Res}(D,x_i)$ induces an endomorphism of $F_{i,j}/F_{i,j+1}$.

\subsection{Connections and Atiyah exact sequence}

As in Section \ref{sec3.2}, we assume that
$\text{par-deg}(L_*)\,=\, 0$. We also fix an
algebraic connection $D_L$ on $L_*$.

Let $E_*\,=\, (E\, ,\{F_{i,j}\}\, ,\{\alpha_{i,j}\})$ be a parabolic
vector bundle, and let $(E_*\, ,\varphi)$ be an
orthogonal or symplectic parabolic bundle. Let $U\,
\subset\, X$ be a nonempty Zariski open subset, and let $w$ be a
function on $U$. Note that $w\cdot D_L$ is a generalized
logarithmic connection on $L\vert_U$ of weight $w$.

If $V$ (respectively, $W$) is a vector bundle on $U$
equipped with a generalized logarithmic connection $D_V$ (respectively, 
$D_W$) of weight $w$, then $D_V\otimes \text{Id}_W+ \text{Id}_V
\otimes D_W$ is a generalized logarithmic connection on
$V\otimes W$ of weight $w$. Also, $D_V$ induces a generalized 
logarithmic connection $D_{V^*}$ on
the dual vector bundle $V^*$ of weight $w$. This
$D_{V^*}$ is uniquely defined by the following identity:
$$
t(D_V(s)) + (D^*_V(t))(s) \,=\, w\cdot d(t(s))\, ,
$$
where $s$ and $t$ are locally defined sections of $V$ and $V^*$
respectively.

Consequently, for any generalized algebraic connection $D$ of
weight $w$ on the
parabolic vector bundle $E_*\vert_U$,
we have a generalized algebraic connection $D^*$ on $E^*_*\vert_U$
of weight $w$. Also, $D^*$ and $w\cdot D_L$ together produce
a generalized algebraic connection on $E^*_*\otimes L_*$ of weight $w$.

A generalized algebraic connection on $D$ of weight $w$ on the
parabolic vector bundle $E_*\vert_U$ is said to be \textit{compatible}
with $\varphi$ if the isomorphism $\widetilde\varphi$ in \eqref{e6}
takes $D$ to the generalized algebraic connection on $E^*_*\otimes
L_*$ constructed using $D$ and $w\cdot D_L$.

A \textit{generalized algebraic connection} on $(E_*\, , \varphi)$
over $U$ is a generalized algebraic connection on
the parabolic vector bundle $E_*$ compatible with $\varphi$.

Let
\begin{equation}\label{f5}
{\mathcal D}(E_*\, ,\varphi)
\end{equation}
be the sheaf on $X$ defined by the generalized algebraic
connections on $(E_*\, ,\varphi)$. 

We will show that ${\mathcal D}(E_*\, ,\varphi)$ defines an
algebraic vector bundle over $X$. For that
purpose, first note that
if $(w_1\, ,D_1)$ and $(w_2\, ,D_2)$ are generalized 
algebraic connections on
$(E_*\, ,\varphi)$ over $U$, then $(w_1+w_2\, ,D_1+D_2)$ is a
generalized algebraic connection on $(E_*\, ,\varphi)$ over
$U$. Also, if $w$ is a function on $U$, then
$(w\cdot w_1\, ,w\cdot D_1)$ is a generalized algebraic
connection on $(E_*\, ,\varphi)$ over $U$. Consequently,
${\mathcal D}(E_*\, ,\varphi)$ defines a coherent sheaf on $X$.
Clearly, the ${\mathcal O}_X$--module ${\mathcal D}(E_*\, ,\varphi)$
is torsionfree, and its rank coincides with $1+\text{rank}(
\text{ad}(E_*\, ,\varphi))$ (see \eqref{ad} for
$\text{ad}(E_*\, ,\varphi)$). Hence ${\mathcal 
D}(E_*\, 
,\varphi)$
is a vector bundle over $X$ of rank $1+\text{rank}(
\text{ad}(E_*\, ,\varphi))$.

Let
\begin{equation}\label{f7}
\eta\, :\, {\mathcal D}(E_*\, ,\varphi)\, \longrightarrow\, {\mathcal 
O}_X
\end{equation}
be the homomorphism defined by $(w\, , D)\, \longmapsto\,
w$. From the ${\mathcal O}_X$--module structure of ${\mathcal D}(E_*\, 
,\varphi)$ described above it follows immediately that $\eta$ is
${\mathcal O}_X$--linear. Also, $\eta$ is surjective.

The sheaf of generalized algebraic connections of weight zero on
$(E_*\, ,\varphi)$ coincides with the sheaf of sections of the
vector bundle $\text{ad}^0(E_*\, ,\varphi)\otimes
\Omega_X$, where $\text{ad}^0(E_*\, ,\varphi)$ is constructed
in \eqref{ad2}. Hence we have an inclusion
$$
\text{ad}^0(E_*\, ,\varphi)\otimes \Omega_X\,\hookrightarrow\,
{\mathcal D}(E_*\, ,\varphi)\, .
$$
Using this inclusion we get a short
exact sequence of vector bundles on $X$
\begin{equation}\label{f6}
0\, \longrightarrow\, \text{ad}^0(E_*\, ,\varphi)\otimes \Omega_X\, 
\longrightarrow\,
{\mathcal D}(E_*)\, \stackrel{\eta}{\longrightarrow}\,
{\mathcal O}_X \, \longrightarrow\, 0\, ,
\end{equation}
where $\eta$ is the homomorphism in \eqref{f7}.

The short exact sequence in \eqref{f6} is a twisted form of the
Atiyah exact sequence. More precisely, if the parabolic
structure on $E_*$ is trivial, then \eqref{f6} tensored with
$TX$ coincides  with the usual Atiyah exact sequence for the
corresponding orthogonal or symplectic bundle. We recall that
an algebraic connection on a principal bundle is an algebraic
splitting of the Atiyah exact sequence associated to the
principal bundle \cite{At}. Also, note that splittings of
a given short exact sequence are in bijective correspondence with
the splittings of the short exact sequence obtained by
tensoring the given exact sequence with some line bundle.

Recall the definition of an algebraic connection on $(E_*\, ,
\varphi)$ (see Definition \ref{def2}). If
$$
\sigma\, :\, {\mathcal O}_X\, \longrightarrow\, {\mathcal D}(E_*\, ,
\varphi)
$$
is a homomorphism such that $\eta\circ\sigma\,=\,
\text{Id}_{{\mathcal O}_X}$, where $\eta$ is the homomorphism
in \eqref{f7}, then the section $\sigma(1)$ of
${\mathcal D}(E_*\, ,\varphi)$ is an algebraic connection on $(E_*\, 
,\varphi)$; here $1$ is the section of ${\mathcal O}_X$ given by
the constant function $1$. Conversely, if $D$ is an algebraic
connection on $(E_*\, ,\varphi)$, then there is a unique homomorphism
$$
\sigma\, :\, {\mathcal O}_X\, \longrightarrow\, {\mathcal D}(E_*\, ,
\varphi)
$$
such that $\sigma(1)\,=\, D$. In other words, algebraic connections on 
$(E_*\, , \varphi)$ are the splitting of the short exact sequence in 
\eqref{f6}.

\section{Criterion for an algebraic connection}

We assume that $\text{par-deg}(L_*)\,=\, 0$. Fix an algebraic
connection $D_L$ on the parabolic line bundle $L_*$. Let
$$
(E_*\, ,\varphi)\,=\, ((E\, ,\{F_{i,j}\}\, ,\{\alpha_{i,j}\})\, ,
\varphi)
$$
be an orthogonal or symplectic parabolic bundle. In this section,
we will give a criterion for $(E_*\, ,\varphi)$ to admit an
algebraic connection.

\subsection{The case of symplectic bundles}

First assume that $\varphi$ is alternating. Let $2r$ be the rank of
$E$.

Take any parabolic vector bundle $V_*$. Using the natural pairing
of $V_*$ with the parabolic dual $V^*_*$, the parabolic
vector bundle $V_*\oplus (V^*_*\otimes L_*)$ is equipped with
a symplectic form with values in $L_*$. Let
$$
\varphi^a_{V_*}\, :\, (V_*\oplus (V^*_*\otimes L_*))\otimes
(V_*\oplus (V^*_*\otimes L_*))\, \longrightarrow\, L_*
$$
be this symplectic form on $V_*\oplus (V^*_*\otimes L_*)$.

\begin{theorem}\label{thm1}
The parabolic symplectic bundle $(E_*\, ,\varphi)$
admits an algebraic connection if and only if the 
following holds: For every parabolic vector bundle $V_*$ satisfying
the condition that there is a symplectic parabolic vector bundle $(W_*\, 
, \phi)$ such that
$$
(E_*\, ,\varphi)\,=\,((V_*\oplus (V^*_*\otimes L_*))\oplus W_*\, , 
\varphi^a_{V_*}\oplus \phi)\, ,
$$
we have
$$
\text{\rm par-deg}(V_*)\,=\, 0\, .
$$
\end{theorem}

\begin{proof}
First assume that $(E_*\, ,\varphi)$ admits an algebraic connection.
Take a parabolic vector bundle $V_*$, and a symplectic 
parabolic vector bundle $(W_*\, , \phi)$, such that
$$
(E_*\, ,\varphi)\,=\,((V_*\oplus (V^*_*\otimes L_*))\oplus W_*\, ,
\varphi^a_{V_*}\oplus \phi)\, .
$$
Fix an isomorphism $\tau\, :\, (E_*\, ,\varphi)\, 
\longrightarrow\,((V_*\oplus (V^*_*\otimes 
L_*))\oplus W_*\, , \varphi^a_{V_*}\oplus \phi)$. Let
\begin{equation}\label{f1}
i_V\, :\, V_*\, \longrightarrow\, (V_*\oplus (V^*_*\otimes
L_*))\oplus W_* ~\,~\text{~and~}~\,~ j_V\, :=\,
(V_*\oplus (V^*_*\otimes L_*))\oplus W_* \, \longrightarrow\, V_*
\end{equation}
be the injection and projection respectively constructed using
$\tau$.

Let $D$ be an algebraic connection on $(E_*\, ,\varphi)$. Consider the 
composition
$$
V_*\, \stackrel{i_V}{\longrightarrow}\, (V_*\oplus (V^*_*\otimes
L_*))\oplus W\, \stackrel{D}{\longrightarrow}\,
((V_*\oplus (V^*_*\otimes
L_*))\oplus W)\otimes \Omega_X\otimes{\mathcal O}_X(S)
$$
$$
\stackrel{j_V\otimes\text{Id}_{\Omega_X\otimes{\mathcal O}_X
(S)}}{\longrightarrow}\,
V_*\otimes \Omega_X\otimes{\mathcal O}_X(S)\, .
$$
It is an algebraic connection on the parabolic vector bundle $V_*$.
Now from
Lemma \ref{lem1} we conclude that $\text{\rm par-deg}(V_*)\,=\, 0$.

To prove the converse, we will first show that it is enough to prove
the converse
under the assumption that $(E_*\, ,\varphi)$ is irreducible, meaning
it does not decompose into a direct sum of symplectic parabolic
vector bundles of positive rank.

We can write $(E_*\, ,\varphi)$ as a direct sum
$$
(E_*\, ,\varphi)\,=\, \bigoplus_{i=1}^n (E^i_*\, ,\varphi^i)\, ,
$$
where each $(E^i_*\, ,\varphi^i)$ is irreducible. If the
condition in the theorem holds for $(E_*\, ,\varphi)$, then it
holds for each $(E^i_*\, ,\varphi^i)$. If each symplectic parabolic 
vector bundle $(E^i_*\, ,\varphi^i)$
has an algebraic connection $D^i$, then $\bigoplus_{i=1}^n D^i$ is
an algebraic connection
on $(E_*\, ,\varphi)$. Therefore, it is enough to prove
the converse under the assumption that $(E_*\, ,\varphi)$ is 
irreducible.

We assume that $(E_*\, ,\varphi)$ is irreducible. Assume that
the condition in the theorem holds for $(E_*\, ,\varphi)$.

We will show that the short exact sequence in \eqref{f6} splits;
recall that any splitting of \eqref{f6} is an
algebraic connection on $(E_*\, ,\varphi)$.

Using Serre duality,
the obstruction to the splitting of \eqref{f6} is a cohomology
class
\begin{equation}\label{c}
c\, \in\, H^1(X,\, \text{ad}^0(E_*\, ,\varphi)\otimes \Omega_X)
\,=\, H^0(X,\,  \text{ad}(E_*\, ,\varphi))^*
\end{equation}
(see \eqref{g1}). We will investigate the functional $c$ of
$H^0(X,\,  \text{ad}(E_*\, ,\varphi))$.

Take any $A\, \in\, H^0(X,\,  \text{ad}(E_*\, ,\varphi))$.
So, $A$ is an endomorphism of $E$ compatible with $\varphi$ which 
preserves the quasi--parabolic flags for $E_*$. Consider the
characteristic polynomial of $A(x)$, $x\, \in \, X$. Since
there are no nonconstant algebraic functions on $X$, the
coefficients of the characteristic polynomial of $A(x)$ are
independent of $x$. Hence the eigenvalues of $A(x)$, along with
their multiplicities, are independent of $x$. Since
$(E_*\, ,\varphi)$ is irreducible, there is exactly one eigenvalue
(otherwise the generalized eigenspace decomposition of $E$
contradicts irreducibility). On the other hand, $A$ is of trace
zero (see \eqref{in}). Hence $A$ does not have any
nonzero eigenvalue.

Since $A$ does not have any  nonzero eigenvalue, we
conclude that the endomorphism $A$ is nilpotent. Now from
the argument in the proof of Proposition 18(ii) in \cite[p. 202]{At}
it follows that the functional $c$ in \eqref{c} satisfies the
identity $c(A)\,=\, 0$. Hence $c\,=\, 0$. Therefore,
$(E_*\, ,\varphi)$ admits an algebraic connection.
\end{proof}

\subsection{The case of orthogonal bundles}

We now consider the case where $(E_*\, ,\varphi)$ is an orthogonal
parabolic vector bundle.

Take any parabolic vector bundle $V_*$. Using the natural pairing
of $V_*$ with the parabolic dual $V^*_*$, the parabolic
vector bundle $V_*\oplus (V^*_*\otimes L_*)$ is equipped with
an orthogonal form with values in $L_*$. Let
$$
\varphi^s_{V_*}\, :\, (V_*\oplus (V^*_*\otimes L_*))\otimes
(V_*\oplus (V^*_*\otimes L_*))\, \longrightarrow\, L_*
$$
be this orthogonal form on $V_*\oplus (V^*_*\otimes L_*)$.

\begin{theorem}\label{thm2}
The orthogonal parabolic
bundle $(E_*\, ,\varphi)$ admits an algebraic connection if and only 
the following holds: For every parabolic vector bundle $V_*$ satisfying
the condition that there is an orthogonal parabolic vector bundle 
$(W_*\, , \phi)$ such that
$$
(E_*\, ,\varphi)\,=\,((V_*\oplus (V^*_*\otimes L_*))\oplus W_*\, , 
\varphi^s_{V_*}\oplus \phi)\, ,
$$
we have
$$
\text{\rm par-deg}(V_*)\,=\, 0\, .
$$
\end{theorem}

The proof of Theorem \ref{thm2} is identical to the proof of
Theorem \ref{thm1}.

In the absence of any parabolic structure, Theorem \ref{thm1}
and Theorem \ref{thm2} coincide with the main theorem of
\cite{AB} (Theorem 4.1) for symplectic and orthogonal bundles
respectively.

\section{Semistable and polystable parabolic bundles}

\subsection{Semistability of tensor product}\label{sec5.1}

Let $E_*\,=\, (E\, ,\{F_{i,j}\}\, ,\{\alpha_{i,j}\})$ be a parabolic
vector bundle over $X$. Any subbundle $F$ of $E$ is equipped with an
induced parabolic structure which is obtained by restricting the 
quasi--parabolic filtrations and the parabolic weights of $E$ to
$F$. Let $F_*$ be the parabolic vector bundle obtained this way.

The parabolic vector bundle $E_*$ is called \textit{stable}
(respectively, \textit{semistable}) if for every subbundle
$F\, \subset\, E$ with $0\, <\, \text{rank}(F)\, <\, \text{rank}(E)$,
the inequality
\begin{equation}\label{s}
\mu_{\rm par}(F_*)\, <\, \mu_{\rm par}(E_*)
~\,~ \text{~(respectively,~}~\,~ \mu_{\rm par}(F_*)\, \leq\,
\mu_{\rm par}(E_*){\rm )}
\end{equation}
holds (see \cite{Se1}, \cite[p. 69, D\'efinition 6]{Se2}).

The parabolic vector bundle $E_*$ is called \textit{polystable}
if is semistable, and isomorphic to a direct sum of stable
parabolic vector bundles.

Fix a complete Hermitian metric $g_X$ on $X\setminus S$;
it is K\"ahler because $\dim_{\mathbb C} X\,=\, 1$.
The notion of a Einstein--Hermitian metric on vector bundles
over $X$ extends to a notion of Einstein--Hermitian metric on
parabolic vector bundles (see \cite[p. 492, Definition 5.7]{Po},
\cite{Si} for the details).

The following is a basic theorem (see \cite[p. 497, Theorem 6.4]{Po}, 
\cite[p. 718, Theorem]{Si}):

\begin{theorem}\label{b}
A parabolic vector bundle $E_*$ admits
an Einstein--Hermitian connection if and only if $E_*$ is
polystable, and the Einstein--Hermitian connection on a
polystable parabolic bundle is unique.
\end{theorem}

A $C^\infty$ connection on $E_*$ is a $C^\infty$ splitting of
the short exact sequence in \eqref{f6}.
An Einstein--Hermitian connection is not algebraic unless it is
flat.

\begin{lemma}\label{lem4}
If $E_*$ and $V_*$ are parabolic polystable vector bundles,
then the parabolic tensor product $E_*\otimes V_*$ is also
polystable.
\end{lemma}

\begin{proof}
Both $E_*$ and $V_*$ admit Einstein--Hermitian connection
by Theorem \ref{b}. The connection on $E_*\otimes V_*$ induced
by Einstein--Hermitian connections on $E_*$ and $V_*$ is again 
Einstein--Hermitian. Hence $E_*\otimes V_*$ is polystable.
\end{proof}

Any semistable parabolic vector bundle admits a filtration of
subbundles such that each successive quotient is polystable of same
parabolic slope. Therefore, Lemma \ref{lem4} has the following
corollary:

\begin{corollary}\label{cor1}
If $E_*$ and $V_*$ are parabolic semistable vector bundles,
then the parabolic tensor product $E_*\otimes V_*$ is also
semistable.
\end{corollary}

\subsection{Semistable and stable orthogonal and symplectic bundles}

Let $V$ be a finite dimensional complex vector space equipped
with an orthogonal form $B$. Let
$\text{SO}(V)\, \subset\, \text{SL}(V)$
be the subgroup consisting of all automorphisms that preserve $B$.
Let
$$
\text{GO}(V)\, \subset\, \text{GL}(V)
$$
be the subgroup consisting of all automorphisms $T$ satisfying the
condition that there is a constant $c\, \in\, {\mathbb C}^*$ such that
$$
B(T(v)\, ,T(w))\,=\, c\cdot B(v\, ,w)
$$
for all $v\, ,w\, \in\, V$. So $\text{GO}(V)$ fits in a short exact 
sequence
$$
e\, \longrightarrow\, \text{SO}(V)\, \longrightarrow\,\text{GO}(V)
\, \longrightarrow\, {\mathbb C}^*\, \longrightarrow\, e\, .
$$

A linear subspace $V_0\, \subset\, V$ is called \textit{isotropic}
if $B(v\, ,w)\,=\,0$ for all $v\, ,w\, \in\, V_0$. The subgroup of
$\text{GO}(V)$ that preserves a fixed nonzero isotropic subspace
$V_0$ is a maximal parabolic subgroup of $\text{GO}(V)$. In fact
all maximal parabolic subgroups of $\text{GO}(V)$ arise this
way. Let $V_0$ be a nonzero isotropic subspace of $V$. Let
$$
P\, \subset\, \text{GO}(V)
$$
be the corresponding maximal parabolic subgroup. We will describe
a Levi subgroup of $P$.
Consider the orthogonal subspace $V^\perp_0\, \subset\, V$
for $V_0$. Since $V_0$ is isotropic, we have $V_0\, \subset\,
V^\perp_0$. Fix a complement $W_0\, \subset\, V^\perp_0$ of the
subspace $V_0$. The restriction of $B$ to $W_0$ is nondegenerate. Fix
an isotropic subspace $V_1\, \subset\, V$ such that $V_1$ is
a complement of $V^\perp_0$. The subgroup of $P$ consisting of
all automorphisms preserving both $W_0$ and $V_1$ is a Levi subgroup
of $P$. All Levi subgroups of $P$ are of this form for some choices
of $W_0$ and $V_1$.

If $B'$ is a symplectic form on $V$, then define
$$
\text{Gp}(V)\, \subset\, \text{GL}(V)
$$
to be the subgroup consisting of all automorphisms $T$ satisfying the
condition that there is $c\, \in\, {\mathbb C}^*$ such that
$$
B'(T(v)\, ,T(w))\,=\, c\cdot B'(v\, ,w)
$$
for all $v\, ,w\, \in\, V$. As before, a linear subspace $V_0\, 
\subset\, V$ is called \textit{isotropic}
if $B(v\, ,w)\,=\,0$ for all $v\, ,w\, \in\, V_0$. Maximal parabolic
subgroups of $\text{Gp}(V)$, and the Levi subgroups of
maximal parabolic subgroups, have exactly identical description as
those for $\text{GO}(V)$.

Let $E$ be a holomorphic vector bundle on $X$, and let $L$ be 
a holomorphic line bundle on $X$. Let
$$
\varphi_0\, :\, E\otimes E\, \longrightarrow\, L
$$
be a nondegenerate bilinear form which is either symmetric or
anti--symmetric. So $(E\, ,\varphi_0)$ defines a principal
$\text{GO}(V)$--bundle or a principal
$\text{Gp}(V)$--bundle on $X$ depending on whether $\varphi_0$
is symmetric or anti--symmetric, where $V$ is as before with
$\dim V\,=\, \text{rank}(E)$.

A holomorphic subbundle $F\, \subset\, E$ is called \textit{isotropic}
if $\varphi_0(F\bigotimes F)\,=\, 0$.

Combining the above descriptions of maximal parabolic and Levi 
subgroups with the definition of a (semi)stable principal bundle
(see \cite[page 129, Definition 1.1]{Ra}
and \cite[page 131, Lemma 2.1]{Ra}), we get the following:

The principal bundle defined by $(E\, ,\varphi_0)$ is
stable (respectively, semistable) if and only if for all
nonzero isotropic subbundles $F\, \subset\, E$,
$$
\frac{\text{degree}(F)}{\text{rank}(F)}\, <\,
\frac{\text{degree}(E)}{\text{rank}(E)}~\, ~\text{(respectively,}~
\frac{\text{degree}(F)}{\text{rank}(F)}\, \leq\,
\frac{\text{degree}(E)}{\text{rank}(E)}{\rm )}\, .
$$

The principal bundle defined by $(E\, ,\varphi_0)$ is 
polystable if and only if
either $(E\, ,\varphi_0)$ is stable, or there is a 
polystable vector bundle $W$ with
$$
\frac{\text{degree}(W)}{\text{rank}(W)}\,=\,
\frac{\text{degree}(E)}{\text{rank}(E)}\, ,
$$
and an orthogonal or symplectic stable parabolic
vector bundle $(F\, ,\phi)$ (depending on whether
$(E\, ,\varphi_0)$ is orthogonal or symplectic),
with $\phi$ taking values in the
same line bundle $L$ as for $\varphi_0$, such that
$(E\, ,\varphi_0)$ is isomorphic to the direct sum
$$
((W\oplus (W^*\otimes L))\oplus F\, , \varphi_{W}
\oplus\phi)\, ,
$$
where $\varphi_{W}$ is the obvious orthogonal or symplectic 
structure on $W\oplus (W^*\otimes L)$ constructed using the 
natural paring of $W$ with its dual $W^*$.

\subsection{Semistable and stable orthogonal and symplectic
parabolic bundles}

Let $(E_*\, ,\varphi)\,=\, ((E\, ,\{F_{i,j}\}\, ,\{\alpha_{i,j}\})\, ,
\varphi)$ be an orthogonal or symplectic parabolic bundle.
To clarify, we no longer assume that the parabolic
degree of $L_*$ is zero.

A subbundle $F$ of $E$ is called \textit{isotropic} if the
restriction of $\varphi$ to $F\otimes F$ vanishes identically.

The bundle $(E_*\, ,\varphi)$ is called \textit{stable}
(respectively, \textit{semistable}) if
$$
\mu_{\rm par}(F_*)\, <\, \mu_{\rm par}(E_*)
~\,~ \text{~(respectively,~}~\,~ \mu_{\rm par}(F_*)\, \leq\, \mu_{\rm 
par}(E_*){\rm )}
$$
for all nonzero isotropic subbundles $F$ of $E$ with
$0\, <\, \text{rank}(F)\, <\, \text{rank}(E)$.

The bundle $(E_*\, ,\varphi)$ is called \textit{polystable} if
either $(E_*\, ,\varphi)$ is stable, or there is a 
parabolic polystable vector
bundle $V_*$ with $\mu_{\rm par}(V_*)\,=\, \mu_{\rm par}(E_*)$, and
an orthogonal or symplectic stable parabolic
vector bundle $(F_*\, ,\phi)$ (depending on whether
$(E_*\, ,\varphi)$ is orthogonal or symplectic) such that
$(E_*\, ,\varphi)$ is isomorphic to the direct sum
$$
((V_*\oplus (V^*_*\otimes L_*))\oplus F_*\, , \varphi_{V_*}
\oplus\phi)\, ,
$$
where $\varphi_{V_*}$ is the obvious orthogonal or symplectic 
structure on $V_*\oplus (V^*_*\otimes L_*)$ constructed using the 
natural paring of $V_*$ with its parabolic dual $V^*_*$
(so $\varphi_{V_*}$ is either $\varphi^a_{V_*}$ in Theorem
\ref{thm1}, or $\varphi^s_{V_*}$ in Theorem \ref{thm2}); the
parabolic vector bundle $F_*$ is allowed to be zero.

It is easy to see that if $(E_*\, ,\varphi)$ is polystable, then
$(E_*\, ,\varphi)$ is semistable.

In the above definition, the condition that $\mu_{\rm par}(V_*)
\,=\,\mu_{\rm par}(E_*)$ implies that
$$
\mu_{\rm par}(V^*_*\otimes L_*)\,=\, \mu_{\rm par}(E_*)
$$
because $\mu_{\rm par}(V_* \oplus(V^*_*\otimes L_*))\,=\,
\mu_{\rm par}(E_*)$ (see \eqref{iii}). If $F_*$ in the
above definition is nonzero, then $\mu_{\rm par}(F_*)\,=\,
\mu_{\rm par}(E_*)$ from \eqref{iii}.

\subsection{Harder--Narasimhan filtration}

Let $(E_*\, ,\varphi)\,=\, ((E\, ,\{F_{i,j}\}\, ,\{\alpha_{i,j}\})\, ,
\varphi)$ be an orthogonal or symplectic parabolic bundle.
Assume that the parabolic vector bundle $E_*$ is not semistable.
Then it has a unique Harder--Narasimhan filtration
\begin{equation}\label{hn}
V^1_*\, \subset\, V^2_* \, \subset\,\cdots \, \subset\,
V^{n-1}_* \, \subset\, V^n_*\,=\, E_*
\end{equation}
(see \cite[p. 70, Th\'eor\`eme 8]{Se2}).
Consider the filtration of $E^*_*\otimes L_*$
\begin{equation}\label{hn2}
E^*_*\otimes L^* \, \twoheadrightarrow\, (V^{n-1}_*)^*\otimes L_*
\, \twoheadrightarrow\,\cdots \, \twoheadrightarrow\, (V^1_*)^*
\otimes L_*\, ,
\end{equation}
where $V^i_*$, $1\leq\, i\, \leq\, n$, are as in \eqref{hn} (the
kernels of the projections from $E^*_*\otimes L^*$ produce the
filtration in \eqref{hn2}). From the definition of a
Harder--Narasimhan filtration it follows immediately
that \eqref{hn2} is the Harder--Narasimhan filtration
of $E^*_*\otimes L_*$.

Any isomorphism between two parabolic vector bundles
preserves the Harder--Narasimhan filtrations,
because the Harder--Narasimhan filtration is unique.
Therefore, the isomorphism $\widetilde\varphi$ in \eqref{e6} takes
the filtration in \eqref{hn} to the filtration in \eqref{hn2}.
Therefore, we have the following proposition:

\begin{proposition}\label{prop1}
For any $i\, \in\, [1\, ,n-1]$, the image $\widetilde{\varphi}(V^i_*)$
coincides with the kernel of the projection
$E^*_*\otimes L^* \, \twoheadrightarrow\, (V^{n-i}_*)^*\otimes L_*$
in \eqref{hn2}.
\end{proposition}

Since $V^1_* \, \subset\, V^{n-1}_*$, Proposition \ref{prop1} has the
following corollary:

\begin{corollary}\label{cor2}
The subbundle $V^1_* \, \subset\, E_*$ in \eqref{hn} is isotropic.
\end{corollary}

\begin{proposition}\label{prop2}
Let $(E_*\, ,\varphi)\,=\, ((E\, ,\{F_{i,j}\}\, ,\{\alpha_{i,j}\})
\, ,\varphi)$ be an orthogonal or symplectic parabolic bundle. Then
$(E_*\, ,\varphi)$ is semistable if and only if the parabolic
vector bundle $E_*$ is semistable.
\end{proposition}

\begin{proof}
If the parabolic vector bundle $E_*$ is semistable, then obviously
$(E_*\, ,\varphi)$ is semistable. To prove the converse, assume
that the parabolic vector bundle $E_*$ is not semistable. Let
\eqref{hn} be the Harder--Narasimhan filtration of $E_*$. Since
$\mu_{\rm par}(V^1_*)\, >\, \mu_{\rm par}(E_*)$, and
$V^1_* \, \subset\, E_*$ is isotropic (see Corollary \ref{cor2}),
we conclude that $V^1_*$ violates the semistability criterion for
$(E_*\, ,\varphi)$. Therefore, $(E_*\, ,\varphi)$ is not semistable.
\end{proof}

\subsection{The socle filtration}

Let $E_*\,=\, (E\, ,\{F_{i,j}\}\, ,\{\alpha_{i,j}\})$ be a semistable
parabolic vector bundle.
If $V_*$ and $W_*$ are polystable nonzero subbundles of
$E_*$ with
$$
\mu_{\rm par}(V_*) \,=\, \mu_{\rm par}(W_*) \,=\,
\mu_{\rm par}(E_*)\, ,
$$
then the parabolic subbundle $F_*\, \subset\, E_*$
generated by $V_*$ and $W_*$ is also polystable with
$\mu_{\rm par}(F_*) \,=\,\mu_{\rm par}(E_*)$; the proof of it
identical to that of \cite[p. 23, Lemma 1.5.5]{HL}.

Therefore, there is a unique maximal polystable parabolic subbundle
$E'_*\, \subset\, E_*$ such that $\mu_{\rm par}(E'_*)\,=\,\mu_{\rm 
par}(E_*)$. This parabolic subbundle $E'_*$ is called the
\textit{socle} of $E_*$.

For the socle $E'_*\, \subset\, E_*$, if $E_*/E'_*\,\not=\, 0$,
then the parabolic vector bundle $E_*/E'_*$ is semistable, and
$\mu_{\rm par}(E_*/E'_*)\,=\,\mu_{\rm par}(E_*)$. We may consider
the socle of $E_*/E'_*$. Hence there is a
unique filtration of parabolic subbundles
\begin{equation}\label{sf}
0\, =\, E^0_*\, \subset\,
E^1_*\, \subset\, E^2_* \, \subset\,\cdots \, \subset\,
E^{m-1}_* \, \subset\, E^m_*\,=\, E_*
\end{equation}
such that for each $i\, \in\, [1\, ,m-1]$, the quotient
$E^i_*/E^{i-1}_*$ is the socle of $E_*/E^{i-1}_*$.

Since $E_*$ is semistable, the parabolic vector bundle $E^*_*\otimes
L_*$ is semistable. The filtration
\begin{equation}\label{sf2}
E^*_*\otimes L^* \, \twoheadrightarrow\, (E^{m-1}_*)^*\otimes L_*
\, \twoheadrightarrow\,\cdots \, \twoheadrightarrow\, (E^1_*)^*
\otimes L_*
\end{equation}
obtained from \eqref{sf} clearly coincides with the socle
filtration of $E^*_*\otimes L^*$.

Let $\varphi$ be a bilinear form $E_*$ such that
$(E_*\, ,\varphi)$ is an orthogonal or symplectic parabolic
vector bundle. From the uniqueness of the socle filtration it
follows immediately that the isomorphism $\widetilde\varphi$
in \eqref{e6} takes the filtration in \eqref{sf} to the
filtration in \eqref{sf2}.

\begin{proposition}\label{prop3}
Let $(E_*\, ,\varphi)$ be a polystable orthogonal or symplectic
parabolic bundle. Then the parabolic vector bundle $E_*$ is
polystable.
\end{proposition}

\begin{proof}
Since $(E_*\, ,\varphi)$ is polystable, it is semistable.
Hence $E_*$ is semistable by Proposition \ref{prop2}.
Let \eqref{sf} be the socle filtration of $E_*$. It was noted
above that the isomorphism $\widetilde\varphi$
in \eqref{e6} takes the filtration in \eqref{sf} to the
filtration in \eqref{sf2}.

We assume that $E_*$ is not polystable. Therefore, we have
$\text{rank}(E^1_*)\, <\, \text{rank}(E_*)$, where $E^1_*$
is the socle of $E_*$ in \eqref{sf}.

Since $\widetilde\varphi$ takes
$E^{1}_*$ isomorphically to the kernel of the projection
$E^*_*\otimes L^* \,\longrightarrow\, (E^{m-1}_*)^*\otimes L_*$,
it follows immediately that
$$
\varphi(E^{1}_*\, , E^{m-1}_*) \, =\, 0\, .
$$
We now conclude that $E^{1}_*$ is isotropic because
$E^{1}_* \,\subset\, E^{m-1}_*$.
Since $(E_*\, ,\varphi)$ is polystable, and $E^1_*$ is an isotropic
subbundle with $\mu_{\rm par}(E^1_*)\,=\, \mu_{\rm par}(E_*)$, we
conclude that there is an orthogonal or symplectic parabolic
vector bundle $(F_*\, ,\phi)$ (depending on whether
$(E_*\, ,\varphi)$ is orthogonal or symplectic) such that
$(E_*\, ,\varphi)$ is isomorphic to the direct sum
$$
((E^1_*\oplus ((E^1_*)^*\otimes L_*))\oplus F_*\, , \varphi_{E^1_*} 
\oplus\phi)\, ,
$$
where $\varphi_{E^1_*}$ is the obvious orthogonal or symplectic 
structure on $E^1_*\oplus ((E^1_*)^*\otimes L_*)$ constructed using the 
natural paring of $E^1_*$ with its parabolic dual $(E^1_*)^*$
(so $\varphi_{E^1_*}$ is either $\varphi^a_{E^1_*}$ in Theorem
\ref{thm1}, or $\varphi^s_{E^1_*}$ in Theorem \ref{thm2}).

Note that $\mu_{\rm par}(E^1_*)\,=\, \mu_{\rm par}((E^1_*)^*\otimes
L_*) \,=\, \mu_{\rm par}(F_*)$. Since the parabolic
vector bundle $E_*$ is isomorphic to $E^1_*\oplus
((E^1_*)^*\otimes L_*)\oplus F_*$, we have a contradiction
to the fact that $E^1_*$ is the maximal polystable parabolic
subbundle of $E_*$ with parabolic slope $\mu_{\rm par}(E_*)$. 
Therefore, we conclude that $E_*$ is polystable.
\end{proof}

\section{Einstein--Hermitian connection on orthogonal
or symplectic parabolic bundle}

As in Section \ref{sec5.1}, fix a complete Hermitian metric $g_X$ on 
$X\setminus S$. Let $\nabla_L$ be the Einstein--Hermitian connection on
$L_*$ given by Theorem \ref{b} (we do not
assume that $\text{par-deg}(L_*)\,=\, 0$).

Let $(E_*\, ,\varphi)\,=\, ((E\, ,\{F_{i,j}\}\, ,\{\alpha_{i,j}\})\, ,
\varphi)$ be an orthogonal or symplectic parabolic bundle; the
form $\varphi$ takes values in $L_*$. If $D$ is
an Einstein--Hermitian connection on $E_*$, then the connection
$D^*$ on $E^*_*$ induced by $D$ is also Einstein--Hermitian. The
Einstein--Hermitian connection $D^*$ on $E^*_*$ and the
Einstein--Hermitian connection $\nabla_L$ on
$L_*$ together define an Einstein--Hermitian connection
on $E^*_*\otimes L_*$.

An \textit{Einstein--Hermitian connection} on $(E_*\, ,\varphi)$ is
an Einstein--Hermitian connection $D$ on $E_*$ such that the
isomorphism $\widetilde \varphi$ in \eqref{e6} takes $D$ to the
connection on $E^*_*\otimes L_*$ constructed using $D^*$
and $\nabla_L$.

\begin{theorem}\label{cor3}
Let $(E_*\, ,\varphi)$ be an orthogonal or symplectic parabolic bundle.
Then $(E_*\, ,\varphi)$ admits an Einstein--Hermitian connection if
and only if $(E_*\, ,\varphi)$ is polystable.
\end{theorem}

\begin{proof}
If $(E_*\, ,\varphi)$ admits an Einstein--Hermitian connection, then
the proof of the ``only if'' part of Theorem \ref{b} gives that
$(E_*\, ,\varphi)$ is polystable. (It should be clarified that the
nontrivial part of Theorem \ref{b} is that a polystable parabolic
bundle admits an Einstein--Hermitian connection.)

To prove the converse, assume that $(E_*\, ,\varphi)$ is polystable.
Then the parabolic vector bundle $E_*$ is polystable by Proposition
\ref{prop3}. Let $D$ be the unique Einstein--Hermitian connection
on $E_*$ (Theorem \ref{b}). The connection on
$E^*_*\otimes L_*$ constructed using $D^*$ and $\nabla_L$ is
Einstein--Hermitian. On the other hand, the connection on
$E^*_*\otimes L_*$ given by $D$ using the isomorphism $\widetilde 
\varphi$ in \eqref{e6} is also Einstein--Hermitian. Now from the 
uniqueness 
of the Einstein--Hermitian connection on a polystable parabolic
vector bundle it follows immediately that the above two connections
on $E^*_*\otimes L_*$ coincide.
\end{proof}

We have the following converse of Proposition \ref{prop3}.

\begin{corollary}\label{cor4}
Let $(E_*\, ,\varphi)$ be an orthogonal or symplectic parabolic
bundle such that $E_*$ is polystable. Then $(E_*\, ,\varphi)$ is 
polystable.
\end{corollary}

\begin{proof}
Let $D$ be the Einstein--Hermitian connection on $E_*$ given
by Theorem \ref{b}. We saw in the proof of Theorem
\ref{cor3} that $D$ is an Einstein--Hermitian connection on
$(E_*\, ,\varphi)$. Therefore, $(E_*\, ,\varphi)$ is polystable
by Theorem \ref{cor3}.
\end{proof}

It should be mentioned that if $(E_*\, ,\varphi)$ is stable,
then $E_*$ is not stable in general. To construct such
examples, take
stable bundles $(E_*\, ,\varphi)$ and $(F_*\, ,\phi)$ with
both orthogonal or both symplectic, and both
$\varphi$ and $\phi$ taking values in a fixed line bundle $L_*$.
Then $(E_*\oplus F_*\, ,\varphi\oplus \phi)$ is also stable,
but $E_*\oplus F_*$ is not stable.


\end{document}